\documentclass[12pt]{amsart}
\usepackage{amsmath,amssymb,amsbsy,amsfonts,amsthm,latexsym,mathabx,
            amsopn,amstext,amsxtra,euscript,amscd,stmaryrd,mathrsfs,
            cite,array,mathtools,enumerate}
\usepackage{url}
\usepackage[colorlinks,linkcolor=blue,anchorcolor=blue,citecolor=blue]{hyperref}

\usepackage{color}
\usepackage{float}

\hypersetup{breaklinks=true}

\AtBeginDocument{%
   \def\MR#1{}
}

\begin{document}

\newtheorem{theorem}{Theorem}
\newtheorem{lemma}[theorem]{Lemma}
\newtheorem{claim}[theorem]{Claim}
\newtheorem{cor}[theorem]{Corollary}
\newtheorem{proposition}[theorem]{Proposition}
\newtheorem{definition}{Definition}
\newtheorem{question}[theorem]{Question}
\newtheorem{remark}[theorem]{Remark}
\newcommand{\hh}{{{\mathrm h}}}

\numberwithin{equation}{section}
\numberwithin{theorem}{section}
\numberwithin{table}{section}

\def\sssum{\mathop{\sum\!\sum\!\sum}}
\def\ssum{\mathop{\sum\ldots \sum}}
\def\dsum{\mathop{\sum \sum}}
\def\iint{\mathop{\int\ldots \int}}

\def\squareforqed{\hbox{\rlap{$\sqcap$}$\sqcup$}}
\def\qed{\ifmmode\squareforqed\else{\unskip\nobreak\hfil
\penalty50\hskip1em\null\nobreak\hfil\squareforqed
\parfillskip=0pt\finalhyphendemerits=0\endgraf}\fi}

\newfont{\teneufm}{eufm10}
\newfont{\seveneufm}{eufm7}
\newfont{\fiveeufm}{eufm5}
%
%
\newfam\eufmfam
     \textfont\eufmfam=\teneufm
\scriptfont\eufmfam=\seveneufm
     \scriptscriptfont\eufmfam=\fiveeufm
%
%
\def\frak#1{{\fam\eufmfam\relax#1}}

\newcommand{\bflambda}{{\boldsymbol{\lambda}}}
\newcommand{\bfmu}{{\boldsymbol{\mu}}}
\newcommand{\bfxi}{{\boldsymbol{\xi}}}
\newcommand{\bfrho}{{\boldsymbol{\rho}}}

\def\fK{\mathfrak K}
\def\fT{\mathfrak{T}}

\def\fA{{\mathfrak A}}
\def\fB{{\mathfrak B}}
\def\fC{{\mathfrak C}}

\def\E{\mathsf {E}}

\def \balpha{\bm{\alpha}}
\def \bbeta{\bm{\beta}}
\def \bgamma{\bm{\gamma}}
\def \blambda{\bm{\lambda}}
\def \bchi{\bm{\chi}}
\def \bphi{\bm{\varphi}}
\def \bpsi{\bm{\psi}}

\def\eqref#1{(\ref{#1})}

\def\vec#1{\mathbf{#1}}


\def\cA{{\mathcal A}}
\def\cB{{\mathcal B}}
\def\cC{{\mathcal C}}
\def\cD{{\mathcal D}}
\def\cE{{\mathcal E}}
\def\cF{{\mathcal F}}
\def\cG{{\mathcal G}}
\def\cH{{\mathcal H}}
\def\cI{{\mathcal I}}
\def\cJ{{\mathcal J}}
\def\cK{{\mathcal K}}
\def\cL{{\mathcal L}}
\def\cM{{\mathcal M}}
\def\cN{{\mathcal N}}
\def\cO{{\mathcal O}}
\def\cP{{\mathcal P}}
\def\cQ{{\mathcal Q}}
\def\cR{{\mathcal R}}
\def\cS{{\mathcal S}}
\def\cT{{\mathcal T}}
\def\cU{{\mathcal U}}
\def\cV{{\mathcal V}}
\def\cW{{\mathcal W}}
\def\cX{{\mathcal X}}
\def\cY{{\mathcal Y}}
\def\cZ{{\mathcal Z}}
\newcommand{\rmod}[1]{\: \mbox{mod} \: #1}

\def\cg{{\mathcal g}}

\def\e{{\mathbf{\,e}}}
\def\ep{{\mathbf{\,e}}_p}
\def\eq{{\mathbf{\,e}}_q}

\def\Tr{{\mathrm{Tr}}}
\def\Nm{{\mathrm{Nm}}}

\def\rE{{\mathrm{E}}}
\def\rT{{\mathrm{T}}}

 \def\SS{{\mathbf{S}}}

\def\lcm{{\mathrm{lcm}}}

\def\t{\tilde}
\def\ov{\overline}

\def\({\left(}
\def\){\right)}
\def\l|{\left|}
\def\r|{\right|}
\def\fl#1{\left\lfloor#1\right\rfloor}
\def\rf#1{\left\lceil#1\right\rceil}
\def\flq#1{\langle #1 \rangle_q}

\def\mand{\qquad \mbox{and} \qquad}

\newcommand{\commIg}[1]{\marginpar{%
\begin{color}{magenta}
\vskip-\baselineskip 
\raggedright\footnotesize
\itshape\hrule \smallskip Ig: #1\par\smallskip\hrule\end{color}}}

\newcommand{\commSi}[1]{\marginpar{%
\begin{color}{blue}
\vskip-\baselineskip 
\raggedright\footnotesize
\itshape\hrule \smallskip Si: #1\par\smallskip\hrule\end{color}}}




\hyphenation{re-pub-lished}

\mathsurround=1pt

\def\bfdefault{b}
\overfullrule=5pt

\def \F{{\mathbb F}}
\def \K{{\mathbb K}}
\def \Z{{\mathbb Z}}
\def \Q{{\mathbb Q}}
\def \R{{\mathbb R}}
\def \C{{\\mathbb C}}
\def\Fp{\F_p}
\def \fp{\Fp^*}

\def\Smn{S_{k,\ell,q}(m,n)}

\def\Kmn{\cK_p(m,n)}
\def\psmn{\psi_p(m,n)}

\def\SM{\cS_{k,\ell,q}(\cM)}
\def\SMN{\cS_{k,\ell,q}(\cM,\cN)}
\def\SAMN{\cS_{k,\ell,q}(\cA;\cM,\cN)}
\def\SABMN{\cS_{k,\ell,q}(\cA,\cB;\cM,\cN)}

\def\SIJq{\cS_{k,\ell,q}(\cI,\cJ)}
\def\SAJq{\cS_{k,\ell,q}(\cA;\cJ)}
\def\SABJq{\cS_{k,\ell,q}(\cA, \cB;\cJ)}

\def\sM{\cS_{k,q}^*(\cM)}
\def\sMN{\cS_{k,q}^*(\cM,\cN)}
\def\sAMN{\cS_{k,q}^*(\cA;\cM,\cN)}
\def\sABMN{\cS_{k,q}^*(\cA,\cB;\cM,\cN)}

\def\sIJq{\cS_{k,q}^*(\cI,\cJ)}
\def\sAJq{\cS_{k,q}^*(\cA;\cJ)}
\def\sABJq{\cS_{k,q}^*(\cA, \cB;\cJ)}
\def\sABJp{\cS_{k,p}^*(\cA, \cB;\cJ)}

 \def \xbar{\overline x}

\author[S.  Macourt] {Simon Macourt}
\address{Department of Pure Mathematics, University of New South Wales,
Sydney, NSW 2052, Australia}
\email{s.macourt@student.unsw.edu.au}

\title{Visible Points on Exponential Curves}

\begin{abstract} 
We provide two new bounds on the number of visible points on exponential curves modulo a prime for all choices of primes. We also provide one new bound on the number of visible points on exponential curves modulo a prime for almost all primes.
\end{abstract}

\maketitle

\section{Introduction}
\subsection{Set up}

We define
\begin{align*}
\cE_{a,g,p}= \{(x,y) : y=ag^x \pmod p \}
\end{align*}
to be the set of points on an exponential modular curve. Furthermore, for real $U,V$ we define $\cE_{a,g,p}(U,V)$ to be the set of points 
$$(x,y)\in \cE_{a,g,p}\bigcap [1,U]\times [1,V].$$
We also define the number of visible points $N_{a,g,p}(U,V)$ to be the number of points for which $(x,y)\in \cE_{a,g,p}(U,V)$ and $\gcd(x,y)=1$. Similarly we define $M_{a,g,p}(U,V)$ to be the number of points for which $(x,y) \in \cE_{a,g,p}(U,V)$.

\subsection{Main Results}
Here we improve previous results (see \eqref{eq:ChShp} below) by giving two bounds using results of Bourgain, Garaev, Konyagin and Shparlinski \cite[Theorems 23 and 24]{BGKS}.
\begin{theorem} \label{thm:vis1}
For $\gcd(a,p)=1$, any $g$ of multiplicative order $t$ modulo $p$ and $U,V \le t$,
\begin{align*}
N_{a,g,p}(U,V) = \frac{6}{\pi^2} \cdot \frac{UV}{p}+ O\left(\left( \frac{U^{3/4}V^{1/4}}{p^{1/8}} + U^{1/4}V^{5/8} \right)p^{o(1)}\right)
\end{align*}
for $U^3V\ge p^{5/2}$.
\end{theorem}

\begin{theorem} \label{thm:vis1-2}
For $\gcd(a,p)=1$, any $g$ of multiplicative order $t$ modulo $p$ and $U,V \le t$,
\begin{align*}
N_{a,g,p}(U,V) = \frac{6}{\pi^2} \cdot \frac{UV}{p}+ O\left(\left(  \frac{U^{6/7}V^{1/7}}{p^{1/28}} + U^{3/13}V^{7/13} \right)p^{o(1)}\right)
\end{align*}
for $U^6V\ge p^{15/4}$.
\end{theorem}



We  also give a new bound for almost all $p$, using \cite[Theorem 31]{BGKS2}.

\begin{theorem} \label{thm:almostp-2}
For sufficiently large positive integers $T, U$ and $V$ and for all but $o(p/\log p)$ primes $p \in [T,2T]$, for any $a$ with $\gcd(a,p)=1$, any $g$ of multiplicative order $t$ modulo $p$ and $U,V \le t$,
\begin{align*}
N_{a,g,p}(U,V) =  \frac{6}{\pi^2} \cdot \frac{UV}{p} + O\left(\left(\frac{U^{2/13}V^{11/13}}{p^{1/26}}+U^{7/22}V^{13/22}\right)p^{o(1)}\right)
\end{align*}
for $U^2V^{11} \ge p^{7}$.
\end{theorem}

\subsection{Comparing Bounds}
We recall the result of Chan and Shparlinski \cite{ChShp}, for $\gcd(a,p)=1$ and any primitive root $g$ modulo $p$,
\begin{align}\label{eq:ChShp}
N_{a,g,p}(U,V)= \frac{6}{\pi^2} \cdot \frac{UV}{p}+O\left( \left( \frac{U^{1/2}V^{1/2}}{p^{1/4}}+\frac{U}{V^{1/35}}+\frac{V}{U^{1/35}} \right) p^{o(1)}\right)
\end{align}
for $1\le U,V \le p-1$ with $UV\ge p^{3/2}$.

We mention that Theorem \ref{thm:vis1} is stronger than \eqref{eq:ChShp} for all $U^3V\ge p^{5/2}$. We also mention that for $U=V$ all our bounds are stronger than the trivial bound 
\begin{align*}
N_{a,g,p}(U,V) \le \min(U,V)
\end{align*}
and of \eqref{eq:ChShp} over their valid regions. We also see  that Theorem \ref{thm:vis1} is always stronger than Theorems \ref{thm:vis1-2} and \ref{thm:almostp-2} for $U=V$ over the regions for which our new bounds are valid.

We notice that Theorem \ref{thm:almostp-2} is strongest for $U$ much larger than $V$. Here we give examples when each result is strongest. One can check that Theorem \ref{thm:vis1} is strongest for $U=V=p^{3/4}$, Theorem \ref{thm:vis1-2} is strongest for $U=p^{3/4}, V=p^{7/8}$ and Theorem \ref{thm:almostp-2} is strongest for $U=p^{5/6}, V=p^{2/3}$.

We also mention that one can get another bound for all $p$ using a result of \cite[Lemma 2.1]{ShpYau}. However, when compared to the bound from Theorem \ref{thm:vis1} one can see that it is trivial. Similarly, one can get another bound for almost all $p$ using Lemma \ref{lem:MI1I2} with $n=2$. Again, comparing this bound with Theorem \ref{thm:vis1} one can see that it is trivial.

\section{Set-up}

We recall the following results given in \cite{BGKS}.
\begin{lemma} \label{lem:MI1I2-Set}
Let $\gcd(a,p)=1$ and $g$ be of multiplicative order $t$ modulo $p$. Let $\cI_1$ and $\cI_2$ be two intervals consisting of $h_1$ and $h_2$ consecutive numbers respectively where $h_2\le t$. Then
\begin{align*}
M_{a,g,p}(\cI_1,\cI_2) < \left(\frac{h_1}{p^{1/3}h_2^{1/6}}+1\right)h_2^{1/2+o(1)}.
\end{align*}
\end{lemma}
\begin{lemma} \label{lem:MI1I2-2}
Let $\gcd(a,p)=1$ and $g$ be of multiplicative order $t$ modulo $p$. Let $\cI_1$ and $\cI_2$ be two intervals consisting of $h_1$ and $h_2$ consecutive numbers respectively where $h_2\le t$. Then
\begin{align*}
M_{a,g,p}(\cI_1,\cI_2) < \left(\frac{h_1}{p^{1/8}h_2^{1/6}}+1\right)h_2^{1/3+o(1)}.
\end{align*}
\end{lemma}

We define $R_{a,g,p}(K;D)$ to be the number of solutions to the congruence 
\begin{align*}
ad \equiv g^d \pmod p, \qquad K+1\le d \le K+D.
\end{align*} 
We also recall the following lemmas given in \cite{ChShp}.
\begin{lemma} \label{lem:MUV}
For $\gcd(ag,p) =1$ and $U,V \le t$ where $t$ is the multiplicative order of $g$ modulo $p$,
\begin{align*}
M_{a,g,p}(U,V)= \frac{UV}{p}+O(p^{1/2}(\log p)^2).
\end{align*}
\end{lemma}
\begin{lemma} \label{lem:RKD}
For $\gcd(ag,p)=1$ and $D\le p$, we have
\begin{align*}
R_{a,g,p}(K;D) \ll D^{1/2}.
\end{align*}
\end{lemma}

We define $K_\nu(p,h,s)$ to be the number of solutions of
\begin{align*}
(x_1+s)\dots (x_\nu +s) \equiv (y_1+s)\dots (y_\nu +s) \nequiv 0 \pmod p
\end{align*} 
where $x_i,y_i \in [1,h]$ for $i=1, \dots, \nu$ and $s \in \F_p$. We recall the following result from  \cite[Theorem 31]{BGKS2}.
\begin{lemma} \label{lem:Kphs}
Let $\nu\ge 1$ be a fixed integer. For sufficiently large positive integers $T>h\ge 3$,
\begin{align*}
K_\nu(p,h,s) \le \left(h^\nu+h^{2\nu -1/2}T^{-1/2}\right)\exp\left(O\left(\frac{\log h}{\log \log h}\right)\right)
\end{align*}
for all $s\in \F_p$ and all but $o(T/\log^2T)$ primes $p\le T$.
\end{lemma}
%
%
We now give the following result. Our proof follows that of \cite[Theorem 23]{BGKS}.
\begin{lemma} \label{lem:MI1I2}
Let $h_1, h_2$ and $T$ be sufficiently large fixed positive integers with $p \in [T,2T]$ and $3\le h_2 \le T$ for some fixed $n\ge 2$, let $g$ be of multiplicative order $t$ modulo $p$ and $\cI_1$ and $\cI_2$ be two intervals consisting of $h_1$ and $h_2$ consecutive integers respectively with $h_1, h_2 \le t$. Then
\begin{align*}
M_{a,g,p}(\cI_1,\cI_2) \le n^{1/(2n)}h_1^{1/(2n)}\left(h_2^{1/2}+h_2^{1-1/(4n)}p^{-1/(4n)}\right)h_2^{o(1)}
\end{align*}
for all but $o(p/\log^2 p)$ primes $p$ with $\gcd(a,p)=1$.
\end{lemma}
\begin{proof}
We recall $M_{a,g,p}(\cI_1,\cI_2)$ is the number of solutions to 
\begin{align}\label{eq:agx}
y \equiv ag^x \pmod p.
\end{align}
Define $\cY \subseteq \cI_2$ to be the values of $y$ which satisfy the congruence \eqref{eq:agx}. Let
\begin{align*}
T(\lambda)= \#\{(y_1,\dots, y_n) \in \cY^n : \lambda \equiv y_1\dots y_n \pmod p\}.
\end{align*}
Therefore,
\begin{align*}
\# \{\lambda: T(\lambda) >0\} \le nh_1
\end{align*}
since
\begin{align*}
\lambda \equiv y_1\dots y_n \equiv a^ng^{x_1+\dots+x_n}.
\end{align*}
By the Cauchy inequality
\begin{align*}
\sum_{\lambda \in \F^*_p} T(\lambda)^2 \ge \frac{1}{nh_1}\left(\sum_{\lambda \in \F^*_p} T(\lambda) \right)^2=\frac{|\cY|^{2n}}{nh_1}.
\end{align*}
Clearly,
\begin{align*}
&\sum_{\lambda \in \F^*_p} T(\lambda)^2 \\
&\quad = \#\{(y_1,\dots,y_n,z_1,\dots,z_n) \in \cY^{2n} : y_1\dots y_n\equiv z_1\dots z_n \pmod p\}  \\
& \quad \le \#\{(y_1,\dots,y_n,z_1,\dots,z_n) \in \cI_2^{2n} : y_1\dots y_n\equiv z_1\dots z_n \pmod p\}.
\end{align*}
Hence, by Lemma \ref{lem:Kphs}
\begin{align*}
 \sum_{\lambda \in \F^*_p} T(\lambda)^2 \le (h_2^{n}+h_2^{2n-1/2}p^{-1/2})h^{o(1)}
\end{align*}
for all but $o(T/\log^2 T)$ primes $p$. Therefore,
\begin{align*}
\frac{|\cY|^{2n}}{nh_1} \le (h_2^{n}+h_2^{2n-1/2}p^{-1/2})h^{o(1)}.
\end{align*}
Rearranging, we complete the proof.
\end{proof}

\section{Proofs of main results}
\subsection{Proof of Theorem \ref{thm:vis1}}
Our proof follows that of \cite[Theorem 1]{ChShp}, however we use Lemma \ref{lem:MI1I2-Set} in place of Lemma 3 from \cite{ChShp}.

From \cite[Equation 3]{ChShp}, we have
\begin{align*}
N_{a,g,p}(U,V) = \Sigma_1+\Sigma_2+\Sigma_3
\end{align*}
where
\begin{equation} \label{eq:sigma123}
\begin{split}
\Sigma_1&= \sum_{\substack{\gcd(d,p)=1 \\ 1\le d\le \delta}}\mu (d)M_{a\bar{d},g^d,p}\left(\frac{U}{d}, \frac{V}{d}\right), \\
\Sigma_2&= \sum_{\substack{\gcd(d,p)=1 \\ \delta\le d\le \Delta}}\mu (d)M_{a\bar{d},g^d,p}\left(\frac{U}{d}, \frac{V}{d}\right), \\
\Sigma_3&=\sum_{\substack{\gcd(d,p)=1 \\ d\ge \Delta}}\mu (d)M_{a\bar{d},g^d,p}\left(\frac{U}{d}, \frac{V}{d}\right),
\end{split}
\end{equation}
for two real parameters $\delta$ and $\Delta$, which will be chosen later. From \cite{ChShp} we see
\begin{align*}
\Sigma_1=\frac{6}{\pi^2} \cdot \frac{UV}{p} + O\left(\frac{UV}{p\delta}+\delta p^{1/2}(\log p)^2\right)
\end{align*}
and
\begin{align*}
\Sigma_3 \ll UV\Delta^{-3/2},
\end{align*}
using Lemma \ref{lem:MUV} and Lemma \ref{lem:RKD} respectively. We now use Lemma \ref{lem:MI1I2}, combined with the triangle inequality, to obtain
\begin{align*}
\Sigma_2 &< \sum_{\substack{\gcd(d,p)=1 \\ \delta\le d\le \Delta}} \left(\frac{U}{p^{1/3}d^{5/6}V^{1/6}}+1\right)\left(\frac{V}{d}\right)^{1/2+o(1)}\\
&\ll \frac{UV^{1/3+o(1)}}{\delta^{1/3}p^{1/3}}+\Delta^{1/2}V^{1/2+o(1)}.
\end{align*}
Therefore,
\begin{equation} \label{eq:NUV}
\begin{split}
&N_{a,g,p}(U,V) - \frac{6}{\pi^2} \cdot \frac{UV}{p} \\ 
&\quad \ll \frac{UV}{p\delta}+\delta p^{1/2+o(1)}+\frac{UV^{1/3+o(1)}}{\delta^{1/3}p^{1/3}}+\Delta^{1/2}V^{1/2+o(1)}+UV\Delta^{-3/2}.
\end{split}
\end{equation}
Now, 
\begin{align*}
\frac{UV}{p\delta} \le \frac{UV^{1/3+o(1)}}{\delta^{1/3}p^{1/3}}
\end{align*}
since $\delta \ge 1$ and $U,V\le p$. We balance the second and third terms in \eqref{eq:NUV} by selecting 
\begin{align} \label{eq:delta}
\delta= \frac{U^{3/4}V^{1/4}}{p^{5/8}}.
\end{align}
For $\delta\ge 1$ we need
\begin{align*}
U^{3}V\ge p^{5/2}.
\end{align*} 
We also balance the fourth and fifth terms in \eqref{eq:NUV} by selecting
\begin{align*}
\Delta = U^{1/2}V^{1/4+o(1)}.
\end{align*}
It is clear that $\delta \le \Delta$, therefore
\begin{align*}
N_{a,g,p}(U,V) - \frac{6}{\pi^2} \cdot \frac{UV}{p}\ll \left( \frac{U^{3/4}V^{1/4}}{p^{1/8}} + U^{1/4}V^{5/8} \right)p^{o(1)}.
\end{align*}
This completes the proof.

\subsection{Proof of Theorem \ref{thm:vis1-2}}
We repeat the above but use Lemma \ref{lem:MI1I2-2} for $\Sigma_2$. Hence,
\begin{align*}
\Sigma_2 &< \sum_{\substack{\gcd(d,p)=1 \\ \delta\le d\le \Delta}} \left(\frac{U}{p^{1/8}d^{5/6}V^{1/6}}+1\right)\left(\frac{V}{d}\right)^{1/3+o(1)}\\
&\ll \left(\frac{UV^{1/6}}{\delta^{1/6} p^{1/8}}+\Delta^{2/3}V^{1/3}\right)p^{o(1)}.
\end{align*}
Therefore,
\begin{equation} \label{eq:NUV-2}
\begin{split}
&N_{a,g,p}(U,V) - \frac{6}{\pi^2} \cdot \frac{UV}{p} \\ 
&\quad \ll \frac{UV}{p\delta}+\delta p^{1/2+o(1)}+\left(\frac{UV^{1/6}}{\delta^{1/6} p^{1/8}}+\Delta^{2/3}V^{1/3}\right)p^{o(1)}+UV\Delta^{-3/2}.
\end{split}
\end{equation}
Now, 
\begin{align*}
\frac{UV}{p\delta} \le \frac{UV^{1/6}}{\delta^{1/6}p^{1/8}}
\end{align*}
since $\delta \ge 1$ and $U,V\le p$. We balance the second and third terms in \eqref{eq:NUV-2} by selecting 
\begin{align} \label{eq:delta-2}
\delta= \frac{U^{6/7}V^{1/7}}{p^{15/28}}.
\end{align}
For $\delta\ge 1$ we need
\begin{align*}
U^{6}V\ge p^{15/4}.
\end{align*} 
We also balance the fourth and fifth terms in \eqref{eq:NUV} by selecting
\begin{align*}
\Delta = U^{6/13}V^{4/13}.
\end{align*}
It is clear that $\delta \le \Delta$, therefore
\begin{align*}
N_{a,g,p}(U,V) - \frac{6}{\pi^2} \cdot \frac{UV}{p}\ll \left( \frac{U^{6/7}V^{1/7}}{p^{1/28}} + U^{3/13}V^{7/13} \right)p^{o(1)}.
\end{align*}
This completes the proof.

\subsection{Proof of Theorem \ref{thm:almostp-2}}
We follow the proof of Theorem \ref{thm:vis1}  picking up after \eqref{eq:sigma123}. We now use Lemma \ref{lem:MI1I2}, taking $n=3$, to obtain
\begin{align*}
\Sigma_2 &\le \sum_{\substack{\gcd(d,p)=1 \\ \delta\le d\le \Delta}} 3^{1/6}\left(\frac{U}{d}\right)^{1/6}\left(\left(\frac{V}{d}\right)^{1/2} +\left(\frac{V}{d}\right)^{11/12}p^{-1/12}\right)p^{o(1)} \\
&\ll U^{1/6}\left(V^{11/12}p^{-1/12}\delta^{-1/12}+\Delta^{1/3} V^{1/2}\right)p^{o(1)}
\end{align*}
for all but $o(p/\log^2p)$ primes $p$. Therefore,
\begin{equation} \label{eq:NUV2-2}
\begin{split}
&N_{a,g,p}(U,V) - \frac{6}{\pi^2} \cdot \frac{UV}{p} \\ 
&\quad \ll \frac{UV}{p\delta}+\delta p^{1/2+o(1)} \\
&\qquad \qquad \qquad  +  U^{1/6}\left(V^{11/12}p^{-1/12}\delta^{-1/12}+\Delta^{1/3} V^{1/2}\right)p^{o(1)}+UV\Delta^{-3/2}.
\end{split}
\end{equation}
We note the first term is dominated by the third. We balance the second and third terms in \eqref{eq:NUV2-2} by selecting
\begin{align*}
\delta = \frac{U^{2/13}V^{11/13}}{p^{7/13}}.
\end{align*}
For $\delta \ge 1$ we need
\begin{align*}
U^2V^{11} \ge p^{7}.
\end{align*}
Similarly, we balance the third and fourth terms by selecting
\begin{align*}
\Delta= U^{5/11}V^{3/11}.
\end{align*}
Clearly $\delta \le \Delta$, therefore
\begin{align*}
N_{a,g,p}(U,V) - \frac{6}{\pi^2} \cdot \frac{UV}{p} \ll \left(U^{2/13}V^{11/13}p^{-1/26}+U^{7/22}V^{13/22}\right)p^{o(1)}
\end{align*}
for all but $o(T/\log^2T)$ primes $p\le T$. This concludes the proof.

\bibliographystyle{amsplain}

\end{document}